\documentclass{article}
\usepackage[utf8]{inputenc}
\usepackage{authblk}
\usepackage{setspace}
\usepackage[margin=1.25in]{geometry}
\usepackage{graphicx}
\graphicspath{ {./figures/} }
\usepackage{subcaption}
\usepackage{amsmath}
\usepackage{amssymb}
\usepackage{amsthm}
\usepackage{thmtools}
\usepackage{tikz}
\usepackage{tikz-cd}
\usepackage[utf8]{inputenc}
\usepackage{bbold}
\usepackage[shortlabels]{enumitem}
\usepackage[hidelinks]{hyperref}      
\usepackage[nameinlink]{cleveref}   
\usepackage{csquotes}
\MakeOuterQuote{"}

\usetikzlibrary{patterns}
\tikzset{
pattern size/.store in=\mcSize, 
pattern size = 5pt,
pattern thickness/.store in=\mcThickness, 
pattern thickness = 0.3pt,
pattern radius/.store in=\mcRadius, 
pattern radius = 1pt}
\makeatletter
\pgfutil@ifundefined{pgf@pattern@name@redstripes}{
\pgfdeclarepatternformonly[\mcThickness,\mcSize]{redstripes}
{\pgfqpoint{0pt}{-\mcThickness}}
{\pgfpoint{\mcSize}{\mcSize}}
{\pgfpoint{\mcSize}{\mcSize}}
{
\pgfsetcolor{\tikz@pattern@color}
\pgfsetlinewidth{\mcThickness}
\pgfpathmoveto{\pgfqpoint{0pt}{\mcSize}}
\pgfpathlineto{\pgfpoint{\mcSize+\mcThickness}{-\mcThickness}}
\pgfusepath{stroke}
}}
\makeatother

\newcommand{\floor}[1]{\left\lfloor #1 \right\rfloor}
\newcommand{\ceil}[1]{\left\lceil #1 \right\rceil}
\newcommand{\prs}[1]{\left( #1 \right)}
\newcommand{\crb}[1]{\left\{ #1 \right\}}

\tikzset{
pattern size/.store in=\mcSize, 
pattern size = 5pt,
pattern thickness/.store in=\mcThickness, 
pattern thickness = 0.3pt,
pattern radius/.store in=\mcRadius, 
pattern radius = 1pt}
\makeatletter
\pgfutil@ifundefined{pgf@pattern@name@_7la7libyc}{
\pgfdeclarepatternformonly[\mcThickness,\mcSize]{_7la7libyc}
{\pgfqpoint{0pt}{0pt}}
{\pgfpoint{\mcSize+\mcThickness}{\mcSize+\mcThickness}}
{\pgfpoint{\mcSize}{\mcSize}}
{
\pgfsetcolor{\tikz@pattern@color}
\pgfsetlinewidth{\mcThickness}
\pgfpathmoveto{\pgfqpoint{0pt}{0pt}}
\pgfpathlineto{\pgfpoint{\mcSize+\mcThickness}{\mcSize+\mcThickness}}
\pgfusepath{stroke}
}}
\makeatother
\tikzset{every picture/.style={line width=0.75pt}}

\theoremstyle{plain}
\newtheorem{theorem}{Theorem}[section]
\newtheorem{lemma}[theorem]{Lemma}
\newtheorem{proposition}[theorem]{Proposition}
\newtheorem{corollary}[theorem]{Corollary}
\newtheorem*{theorem*}{Theorem}

\theoremstyle{definition}
\newtheorem{definition}[theorem]{Definition}
\newtheorem{example}[theorem]{Example}

\theoremstyle{remark}
\newtheorem{remark}[theorem]{Remark}

\usepackage[backend=bibtex, style=alphabetic, citestyle=alphabetic]{biblatex}
\newcommand{\Z}{\mathbb{Z}}
\newcommand{\N}{\mathbb{N}}

\newcommand{\Q}{\mathbb{Q}}
\newcommand{\di}{dimension}

\title{Collapsibility in Multiparametric Models of Random Simplicial Complexes}

\author{Jon V. Kogan\footnote{Department of Mathematics, Hebrew University, Jerusalem 91904, Israel. e-mail:
jonatan.kogan$@$mail.huji.ac.il. Supported by ERC grant 3012006831 and ISF grant 3011003773}}

\date{\vspace{-5ex}}

\begin{document}

\maketitle


\begin{abstract}
We study collapsibility in the multiparametric models of random simplicial complexes, namely the lower and upper models.
In the upper model, we improve upon a result of Farber and Nowik, and assert that the homology  is a.a.s concentrated in a single dimension by proving that the complex collapses to that \di.
In the lower model, we prove that the complex a.a.s collapses to the \di\ with maximal non-trivial cohomology. We then compare this threshold to the ones derived previously for the special cases of the clique complex (by Kahle) and the Linial-Meshulam model.

\end{abstract}

\section{Introduction}
\subsection{Background}
Random simplicial complexes form a probabilistic framework for studying topology, extending the classical Erd\H{o}s--R\'enyi random graph model to higher dimensions. 
Foundational examples include the Linial--Meshulam complex $Y_d(n,p)$ (with the 2-\di al version first appearing in \cite{LinialMeshulam}, and in \cite{MeshulamWallach} for general $d$), in which $d$-simplexes are independently added to a complete $(d-1)$-skeleton, and the random clique complex $X(n,p)$, obtained by taking the clique (also called flag) complex of $G(n,p)$ (first appearing in \cite{KAHLE09}). 
Since then, extensive work has been devoted to understanding how topological invariants such as connectivity, homology, and collapsibility depend on the underlying probability parameters; see, for example, \cite{MeshulamWallach,LinialPeled16,AronshtamLinialLM13,MALEN23}. 

These classical models, as well as most others at the time, derive their randomness from faces of a single dimension. To capture interactions between simplexes of different dimensions, Costa and Farber introduced the multiparametric models of random simplicial complexes in \cite{CostaFarber} and \cite{FarberMeadNowik19}:
\begin{definition}\label{defLMM}
Consider a random hypergraph $X(n;p_1,p_2,...)$ composed of independently chosen facets of each \di , where facets of \di \ $i$ are chosen with probability $p_i$.
We may then convert this distribution on hypergraphs to two distributions on simplicial complexes:

Sample a hypergraph $H$ from $X(n;p_1,p_2,...)$ and let $\underline{X}$ be the  maximal simplicial complex contained in it. 
The resulting distribution on simplicial complexes is called the \textbf{Lower Multiparametric Model} or the \textbf{Lower Model} for short, and denoted by $\underline{X}(n;p_1,p_2,...)$.

If, instead, we consider the distribution on $\overline{H}$, the minimal simplicial complex containing $H$, we denote this distribution by
$\overline{X}(n;p_1,p_2,...)$ and call it the \textbf{Upper Multiparametric Model} (or the \textbf{Upper Model} for short).
 \end{definition}
  These may be thought of in the following way:
  $\underline{X}(n;p_1,p_2,...)$ is the distribution on simplicial complexes with $n$ vertices, where each edge is included independently with probability $p_1$, each triangle whose boundary formed in the first step is filled independently with probability $p_2$ and so on.
  $\overline{X}(n;p_1,p_2,...)$, may be thought of in the same way, but without requiring the boundary to exist before attaching a simplex.

 $\underline{X}(n;p_1,p_2,...)$ interpolates between the clique complex and the Linial-Meshulam-Wallach model (with parameters $(p,1,1,...),(1,1,...,1,p_r,0,...)$ respectively), while $\overline{X}(n;p_1,p_2,...)$ is its obvious "dual".

In this paper we wish to study the \emph{collapsibility} of these models. A \textbf{collapse} is the erasing of a pair of faces $\sigma\subset\tau$, where $\sigma$ is a free face of $\tau$.
A complex $C$ \textbf{collapses to \di\ $d$} if there exists a sequence of collapses changing $C$ to a complex of \di\ at most $d$ (see \Cref{def:collapse} for the precise formulation).
A collapse is in particular a combinatorially tractable version of a homotopy equivalence, shrinking the size of the complex while preserving all homotopical invariants (such as homology).

Collapsibility has been extensively studied in the two classical models of random complexes mentioned above. In \cite{AronshtamLinialLM13} it was proven that the threshold for collapsibility of $Y_d(n,p)$ to \di\ $d-1$ is of the form $p=c/n$. Interestingly, this threshold occurs for a constant  strictly smaller than that for the appearance of $d$-dimensional homology (see \cite{LinialPeled16}). 

The analogous result regarding the clique complex $X(n,p)$ is that it collapses onto \di\ $k-1$ for $p=n^{-a},a>1/k$, proven in \cite{MALEN23}.

In this paper, we wish to establish similar collapsibility results for the multiparametric models. Such a result already exists for the upper model, but is not tight. 
We recall the insightful notation from \cite{FarberNowik23} in \Cref{FNNotation}. It will be useful both in formulating their theorem, as well as for the proofs in this paper. We also denote by $E(A\subset X)$ the expected number of appearances of $A$ as a subcomplex of a random complex $X$, or simply $E(A)$ if $X$ is clear from context (see \Cref{defPartial} for the precise definition).
\begin{definition}\label{FNNotation}
    In $\overline{X}(n;p_1,p_2,...)$, where $p_i=n^{-\alpha_i}$:
    \begin{itemize}
        \item $\beta_i:=i+1-\alpha_i\sim \log_n E(\Delta^i\subset X)$, $\beta:=\max\{\beta_i\}$, $l=\lfloor\beta \rfloor$.
        \item $\gamma_k:=\max_{i\ge k}\{\beta_i\}\sim \log_n E(\Delta^k\subset \overline{X})$.
        \item $\nu_k:=2\gamma_k-k$,  $l':=\max\{k| \nu_k\ge 0\}\le \lfloor 2\beta \rfloor\le 2l+1$.
        \item $e_k:=\gamma_k-k=\nu_k-\gamma_k <0$ where the inequality is when $k>l$. $e_{k+1}\le e_k-1$.
    \end{itemize}
    (some facts are included for intuition's sake).
    \end{definition}
Formulated in this notation, Farber and Nowik proved the following:
\begin{theorem}\label{FNResults}
    Let $p_i=n^{-\alpha_i}$ with $p_j=0$ when $j>r$, and $\crb{\alpha_i}_{i=1}^r$ such that $\beta\notin \Z$. Then $\overline{X}(n;p_1,p_2,...)$ a.a.s has a full $(l-1)$-skeleton, collapses onto \di \ $l'$,
    has $H_i(\overline{X};\Z)=0,i<l$,
    and has $\dim(H_l(\overline{X};\Q))=\theta(n^\beta)$. Further, a.a.s $\dim(H_k(\overline{X};\Q))\le \omega(n) n^{\nu_k}$ for $k>l$ and any $\omega(n)\xrightarrow{n\to \infty} \infty$.
\end{theorem}
Thus, in \cite{FarberNowik23}, $\overline{X}$ is known to have non-trivial homology in \di\ $l$, but the ranks of $H_{k}(\overline{X}),l<k\le l'$ are only bounded from above, and collapsibility is known only to \di\ $l'$.
We improve upon this in \Cref{thm:uppercollapse}.

Regarding the lower model $\underline{X}$, no collapsibility result exists in the literature. However, we do have the following theorem from \cite{Fowler19}, effectively bounding from below the possible \di s to collapse to:
\begin{theorem}[Fowler]\label{thmFowler}
    For $\underline{X}(n;p_1,p_2,...)$ with $p_i=n^{-\alpha_i}$, denote: $$S_1^k(\{\alpha_i\}_i)=\sum_{i=1}^{k+1}\binom{k+1}{i}\alpha_i\qquad S_2^k(\{\alpha_i\}_i)= \sum_{i=1}^{k}\binom{k+2}{i+1}\alpha_i.$$
        Then:
        \begin{itemize}
            \item If $S_1^k(\{\alpha_i\}_i)<1$, then $H^k(\underline{X},\Q)=0$ a.a.s (see \Cref{aas}).
            \item If $S_2^k(\{\alpha_i\}_i)>k+2$, then $H^k(\underline{X},\Z)=0$ a.a.s.
            \item If $S_1^k(\{\alpha_i\}_i)\geq 1$ and $S_2^k(\{\alpha_i\}_i)<k+2$ and all $\alpha_i$ are positive, then $H^k(\underline{X},\Q)\neq 0$ a.a.s.
        \end{itemize}
\end{theorem} 
\noindent This bound turns out to be the best possible.

\noindent Note that $S_1^k,S_2^k$ are non-decreasing in $k$.

\subsection{Results} 
In the lower model, we prove:
\begin{theorem}\label{thm:lowerCollapse}
    For $\underline{X}(n;p_1,...)$ where $p_i=n^{-\alpha_i},\alpha_i\ge 0$, let $d$ be the smallest integer s.t $S_2 ^d(\crb{\alpha_i}_i)>d+2$. Then $\underline{X}$ a.a.s collapses onto \di\ $d-1$.
\end{theorem}
In other words, the lower model a.a.s collapses to the largest \di \ which is known to have non-trivial cohomology (which $S_2 ^{d-1}(\crb{\alpha_i}_i)<d+1$ would imply by \Cref{lem:S2entS1} and \Cref{thmFowler}).
The case where $S_2 ^d(\crb{\alpha_i}_i)=d+2$ is where the threshold lies for collapsibility of the Linial-Meshulam-Wallach model (see \Cref{exm:lower} for more detail), so \Cref{thm:lowerCollapse} is the best possible for this level of asymptotic "granularity" and general parameters (see \Cref{ex:clique} for an example where restricting the parameters to the ones of the clique complex yields a strictly better bound).

In order to prove a corresponding result in the upper model, we define in \Cref{def:redun} the concept of \textbf{redundancy}. It is a measure of how vertices are shared among maximal faces, and in \Cref{lem:redunColapse} we show it is strongly related to the \di\ to which a simplicial complex collapses.

Using this, formulated in the notation of \Cref{FNNotation}, we prove that simplexes of \di \ greater than $l$ cannot combine to form sufficiently complex structures, and as a result:
\begin{theorem}\label{thm:uppercollapse}
    Suppose $p_i=n^{-\alpha_i},\alpha_i\ge0$, such that there exists $r\in \N$ s.t $p_j=0,j>r$. Then $\overline{X}(n;p_1,...)$ a.a.s collapses onto \di \ $l=\floor{\max_i(i+1-\alpha_i)}$.
\end{theorem}
 In particular, all homology above \di\ $l$ is strictly $0$ (over any coefficient group), improving the bounds and the collapsibility result from \Cref{FNResults}.
 \cite{FarberNowik23} also prove that $\overline{X} (n;p_1,...)$ a.a.s has a complete ($l-1$)-skeleton, making it appear similar to the Linial-Meshulam-Wallach model after the collapse.
 One should note, however, that collapsing procedures involve choices, as well as the fact that different $l$-faces in $\overline{X}$ are independent only in the case $l=r$, but otherwise are positively correlated if they share at least $r-l$ vertices (as the existence of one makes the existence of all faces containing it more likely).
 From \Cref{thm:uppercollapse} we conclude in \Cref{cor:wedgeSpheres} that, provided $3\le \beta\not \in\N$, $\overline{X}$ is homotopy equivalent to a wedge of spheres.
\subsection{Structure of the paper}
In \Cref{prelim} we establish the basic definitions and lemmas we will use in our proofs. In \Cref{chap:lower} we prove \Cref{thm:lowerCollapse}, as well as compare it to the existing results in the clique and Linial-Meshulam models.
In \Cref{chap:Upper} we prove \Cref{thm:uppercollapse}.
Finally, \Cref{discus} is reserved for discussion.

\section{Preliminaries}\label{prelim}

\begin{definition}\label{aas}    
 Let $D(n;\vec{p}(n))$ be a family of distributions depending on a natural number $n$ and perhaps other parameters. $D(n;\vec{p}(n))$ satisfies a property $q$ \textbf{a.a.s} (asymptotically almost surely) if $\underset{n\to \infty}{\lim} P(q)=1$.
 \end{definition}

\begin{example}
    
Let $G(n,p)$ be a distribution on graphs with $n$ vertices, where each edge appears independently with probability $p$.
A theorem by Erdős and Rényi states that $\frac{\ln(n)}{n}$ is the threshold function for connectivity, meaning that for $p(n)=o(\frac{\ln(n)}{n})$ the graph is a.a.s disconnected, and for $p(n)=\omega(\frac{\ln(n)}{n})$ it is a.a.s connected. \end{example} 

\begin{definition}
    A simplicial complex is called \textbf{pure $d$-dimensional}  if any face in the complex is contained in a $d$-\di al simplex. Note that this implies that $d$ is the maximal dimension.
\end{definition}
\begin{definition}\label{defStrongC}
 For a simplicial complex define a relation $\sim$ on its $d$-dimensional simplexes by having $\sigma \sim \tau$ if they share a $(d-1)$-dimensional face. 
 This is obviously reflexive and symmetric. Complete this relation under transitivity and call it $\sim'$. A \textbf{Strong Connectivity Component} is an equivalence class of $d$-dimensional faces under this relation. 
 
A pure-$d$-dimensional simplicial complex is called \textbf{Strongly Connected} if all of its $d$-\di al faces are in the same component. A simplicial complex where all maximal simplexes are of \di \ greater or equal to $d$ is called \textbf{Strongly Connected With Respect To (w.r.t) Dimension $d$} if its $d$-skeleton is strongly connected.
\end{definition}
The importance of strong connectivity components is seen in the following:
\begin{lemma}[Lemma 2.7 in \cite{kogan24}]\label{lemCupSteenConc}
Let $X$ be a simplicial complex, $d\in\N$, and $\sigma\in A_p$, the Steenrod algebra for prime $p$.
Then any one of $H^d(X)$,  $Im (\cup : H^*(X)\times H^*(X)\to H^d(X))$  and $Im(\sigma)\subset H^d$ is non-zero only if there exists a $d$-\di al strong-connectivity component on which it is non-zero. 
    If $X$ has $d$ as the maximal \di , this is an if and only if for $H^d(X)$.
\end{lemma}

It turns out strongly connected complexes can be constructed one face at a time:

\begin{definition}\label{expansionOp}
    For a simplicial complex $X$ and $d$ s.t $X$ is strongly connected w.r.t \di \ $d$, an \textbf{expansion operation} is the addition of a simplex of \di \ greater or equal to $d$, while keeping the component strongly connected.
    Specifically, the added simplex must share a $(d-1)$-dimensional face with the existing complex, in addition to possibly other faces.\par
    Two expansion operations $A\mapsto A\cup_Q \Delta^D, B\mapsto B\cup_W \Delta^D$ are of the same type if $Q\cong W$. 
    In particular, an expansion operation $A\mapsto A\cup_Q \Delta^D$ is called a \textbf{vertex adding operation} if $Q=\Delta^{d-1}$.
    
\end{definition} 

\begin{lemma}\label{vertBeforeAles}
     For $C$ strongly connected w.r.t \di\ $d$, there exists a sequence $\Delta^{D_0}=C_0\subset C_1\subset...\subset C_m=C$ of expansion operations, such that all $C_i$ are strongly connected w.r.t \di\ $d$, and $C_{i+1}=C_i\cup_{Q_i} \Delta^{D_i}$ where for all $i$, $D_i\ge d$, $Q_i$ contains a $(d-1)$-face and $C_i$ has $i+1$ maximal faces.
    Such a sequence exists for any choice of $C_0$ and is called a \textbf{growing process} or \textbf{growth process} of $C$.
\end{lemma}
\begin{proof}
 Define a graph $G_C=(V_C,E_C)$, where $V_C$ is the set of maximal faces of $C$, and $(\mu,M)$ is an edge if and only if $\dim(\mu\cap M)\ge d-1$. 
    The strong connectivity of $C$ is equivalent to $G_C$ being connected, and so $G_C$ contains a spanning tree.
    Any tree has a filtration $\{v\}=T_0\subset T_1\subset ...\subset T_m=T$, where we add 1 vertex and 1 edge at a time, keeping the tree connected, and such a filtration corresponds to the desired $\Delta^{D_0}=C_0\subset C_1\subset...\subset C_m=C$. Furthermore, $C_0$ may be chosen to be any face, as any vertex may be chosen to be the root of the tree.
\end{proof}

\begin{definition}\label{defPartial}
    For a simplicial complex $A$ with $a$ vertices, and a model of random simplicial complexes $X$, we denote 
    $$E(A\subset X)=\sum_{\underset{|S|=a}{S\subset V(X)}}P(A\subset X|_S),$$
    where $P(A\subset X|_S)$ is the probability of it being possible to map $A$ injectively to the induced subcomplex on $S$.
    In other words, it is the expected number of appearances of $A$ as a subcomplex of $X$. We sometimes simply write "$E(A)$" if $X$ is clear from context.
\end{definition}
In order to use the first and second moment methods, we wish to have terminology tying a growing process of a component to its expectation.
\begin{definition}\label{defBudgetCost}
    For a model of random simplicial complexes $X$, the \textbf{budget} (w.r.t \di\ $d$) is $\lim_{n\to \infty}\log_n(E(\Delta^d\subset X))$.
    For an expansion operation $A\mapsto A'$ the \textbf{cost} is $\lim_{n\to \infty}\log_n\prs{\frac{E(A\subset X)}{E(A'\subset X)}}$. 
    
\end{definition}

\begin{example}\label{expDim2}
    
 The following are the expansion operations for the pure 2-dimensional case:
\begin{center}
\begin{tabular}{| c | c | c |} 
 \hline
 Operation & Cost in $\underline{X}$ & Homotopical effect\\
 \hline
\begin{tikzpicture}[x=0.75pt,y=0.75pt,yscale=-1,xscale=1]

\draw  [fill={rgb, 255:red, 155; green, 155; blue, 155 }  ,fill opacity=1 ] (16.25,54.25) -- (56.25,34.25) -- (56.25,74.25) -- cycle ;
\draw  [color={rgb, 255:red, 208; green, 2; blue, 27 }  ,draw opacity=1 ][pattern=redstripes,pattern size=6pt,pattern thickness=0.75pt,pattern radius=0pt, pattern color={rgb, 255:red, 208; green, 2; blue, 27}] (96.25,54.25) -- (56.25,74.25) -- (56.25,34.25) -- cycle ;
\draw  [fill={rgb, 255:red, 208; green, 2; blue, 27 }  ,fill opacity=1 ] (92,54.25) .. controls (92,51.9) and (93.9,50) .. (96.25,50) .. controls (98.6,50) and (100.5,51.9) .. (100.5,54.25) .. controls (100.5,56.6) and (98.6,58.5) .. (96.25,58.5) .. controls (93.9,58.5) and (92,56.6) .. (92,54.25) -- cycle ;
\end{tikzpicture} &   $2\alpha_1+\alpha_2-1 >0$  &     None \\
 \hline 
\begin{tikzpicture}[x=0.75pt,y=0.75pt,yscale=-1,xscale=1]

\draw  [fill={rgb, 255:red, 155; green, 155; blue, 155 }  ,fill opacity=1 ] (16.25,54.25) -- (56.25,34.25) -- (56.25,74.25) -- cycle ;
\draw  [color={rgb, 255:red, 208; green, 2; blue, 27 }  ,draw opacity=1 ][pattern=redstripes,pattern size=6pt,pattern thickness=0.75pt,pattern radius=0pt, pattern color={rgb, 255:red, 208; green, 2; blue, 27}] (96.25,54.25) -- (56.25,74.25) -- (56.25,34.25) -- cycle ;
\draw  [fill={rgb, 255:red, 0; green, 0; blue, 0 }  ,fill opacity=1 ] (92,54.25) .. controls (92,51.9) and (93.9,50) .. (96.25,50) .. controls (98.6,50) and (100.5,51.9) .. (100.5,54.25) .. controls (100.5,56.6) and (98.6,58.5) .. (96.25,58.5) .. controls (93.9,58.5) and (92,56.6) .. (92,54.25) -- cycle ;
\end{tikzpicture} & $2\alpha_1+\alpha_2>1$ & Adds generator to $H^1$\\
 \hline 
\vspace*{3 pt}
\begin{tikzpicture}[x=0.75pt,y=0.75pt,yscale=-1,xscale=1,rotate=-90]

\draw  [color={rgb, 255:red, 208; green, 2; blue, 27 }  ,draw opacity=1 ][pattern=redstripes,pattern size=6pt,pattern thickness=0.75pt,pattern radius=0pt, pattern color={rgb, 255:red, 208; green, 2; blue, 27}] (28.93,81.57) -- (63.57,61.57) -- (63.57,101.57) -- cycle ;
\draw  [fill={rgb, 255:red, 155; green, 155; blue, 155 }  ,fill opacity=1 ] (26.25,36.93) -- (63.57,61.57) -- (28.93,81.57) -- cycle ;
\draw  [fill={rgb, 255:red, 155; green, 155; blue, 155 }  ,fill opacity=1 ] (26.25,126.21) -- (28.93,81.57) -- (63.57,101.57) -- cycle ;
\end{tikzpicture} & 
    $\alpha_1+\alpha_2>\frac{1}{2}$
& 
    None
\\
 \hline 
 \begin{tikzpicture}[x=0.75pt,y=0.75pt,yscale=-0.75,xscale=0.75]

\draw  [color={rgb, 255:red, 208; green, 2; blue, 27 }  ,draw opacity=1 ][pattern=redstripes,pattern size=6pt,pattern thickness=0.75pt,pattern radius=0pt, pattern color={rgb, 255:red, 208; green, 2; blue, 27}] (28.93,81.57) -- (63.57,61.57) -- (63.57,101.57) -- cycle ;
\draw  [fill={rgb, 255:red, 155; green, 155; blue, 155 }  ,fill opacity=1 ] (26.25,36.93) -- (63.57,61.57) -- (28.93,81.57) -- cycle ;
\draw  [fill={rgb, 255:red, 155; green, 155; blue, 155 }  ,fill opacity=1 ] (26.25,126.21) -- (28.93,81.57) -- (63.57,101.57) -- cycle ;
\draw  [fill={rgb, 255:red, 155; green, 155; blue, 155 }  ,fill opacity=1 ] (103.57,81.57) -- (63.57,101.57) -- (63.57,61.57) -- cycle ;

\end{tikzpicture} & \centering $\alpha_2$ &     Reduces $H^1$ or increases $H^2$ \\
 \hline
\end{tabular}

\end{center}\par
\noindent where the costs are for the lower model where $p_i=n^{-\alpha_i}$, and the inequalities follow from $S^1 _1(\{\alpha_i\}_i)>1$. The black and gray simplexes denote the complex before the expansion operation,  while the red patterned simplexes denote new simplexes. \par
For instance, consider the following growth process for the tetrahedron:
$$\begin{tikzpicture}

\draw[fill=gray,fill opacity=0.4]  (0,0) 
  -- (2,0) 
  -- (2.2,1) 
  -- cycle;

\draw (0,0) node[left]{$A$}
(2,0) node[right]{$B$}
  (2.2,1) node[right]{$C$};
\end{tikzpicture}
\begin{tikzpicture}

\draw[fill=gray,fill opacity=0.4]  (0,0) 
  -- (2,0) 
  -- (2.2,1) 
  -- cycle;
\draw[pattern=redstripes,pattern size=6pt,pattern thickness=0.75pt,pattern radius=0pt, pattern color={rgb, 255:red, 208; green, 2; blue, 27}]  (0,0) 
  -- (1,1.6) 
  -- (2.2,1) 
  -- cycle;
\draw (0,0) node[left]{$A$}
(1,1.6) node[above]{$D$}
(2,0) node[right]{$B$}
  (2.2,1) node[right]{$C$};
\end{tikzpicture}
\begin{tikzpicture}

\draw[fill=gray,fill opacity=0.3]  (0,0) 
  -- (2,0) 
  -- (2.2,1) 
  -- cycle;
\draw[fill=gray,fill opacity=0.3]  (0,0) 
  -- (1,1.6) 
  -- (2.2,1) 
  -- cycle;
\draw (0,0) node[left]{$A$}
(1,1.6) node[above]{$D$}
(2,0) node[right]{$B$}
  (2.2,1) node[right]{$C$};
\draw[pattern=redstripes,pattern size=6pt,pattern thickness=0.75pt,pattern radius=0pt, pattern color={rgb, 255:red, 208; green, 2; blue, 27}]
(0,0) 
  -- (1,1.6) 
  -- (2,0)
  -- cycle;
\end{tikzpicture}
\begin{tikzpicture}
\draw[fill=gray,fill opacity=0.2]  (0,0) 
  -- (1,1.6) 
  -- (2,0)
  -- cycle;

\draw[fill=gray,fill opacity=0.3]  (0,0) 
  -- (2,0) 
  -- (2.2,1) 
  -- cycle;
\draw[fill=gray,fill opacity=0.3]  (0,0) 
  -- (1,1.6) 
  -- (2.2,1) 
  -- cycle;
\draw (0,0) node[left]{$A$}
(1,1.6) node[above]{$D$}
(2,0) node[right]{$B$}
  (2.2,1) node[right]{$C$};
\draw[pattern=redstripes,pattern size=6pt,pattern thickness=0.75pt,pattern radius=0pt, pattern color={rgb, 255:red, 208; green, 2; blue, 27}] (1,1.6) 
  -- (2,0) 
  -- (2.2,1) 
  -- cycle; 
\end{tikzpicture}
$$
So we begin from the triangle $ABC$ and a budget of $3-3\alpha_1-\alpha_2$, and proceed to add vertex $D$ via the first operation, then add triangle $ABD$ via the third operation (costing $\alpha_1+\alpha_2$), and finally "close" the tetrahedron by adding triangle $BCD$ using the fourth operation (costing $\alpha_2$). 
All in all, there remains of the budget $4-6\alpha_1-4\alpha_2$, which, assuming $S^1 _1(\{\alpha_i\}_i)\geq 1$ is less than or equal to $1-\alpha_2$. That is the expected number of subcomplexes isomorphic to the tetrahedron is $O(n^{1-\alpha_2})$.
 \end{example}
We now transition to talking about collapsibility:
\begin{definition}\label{def:collapse}
    Let $C$ be a simplicial complex that has a pair of faces $\sigma\subset \tau$ such that $\sigma$ is a maximal facet of $\tau$, and no simplex contains $\sigma$ apart from $\tau$. $\sigma$ is then called a \textbf{free face} (of either $\tau$ or $C$). 
    Erasing both $\sigma$ and $\tau$ from the complex is called an \textbf{elementary collapse} and induces a homotopy equivalence between the old and new complex.
    $C$ \textbf{collapses onto \di \ $d$} if there exists a series of elementary collapses starting with $C$ and resulting in a simplicial complex with only simplexes of \di \ $d$ or lower.
    $C$ is called \textbf{collapsible} if it collapses onto a point. 
\end{definition}
 Collapsibility is a stronger condition than being null-homotopic. The following is a classic example:
\begin{example}
    This is a triangulation of the (topological) Dunce Hat:
    \begin{center}
        \includegraphics[trim={0 0 0 0cm},clip,scale=0.5]{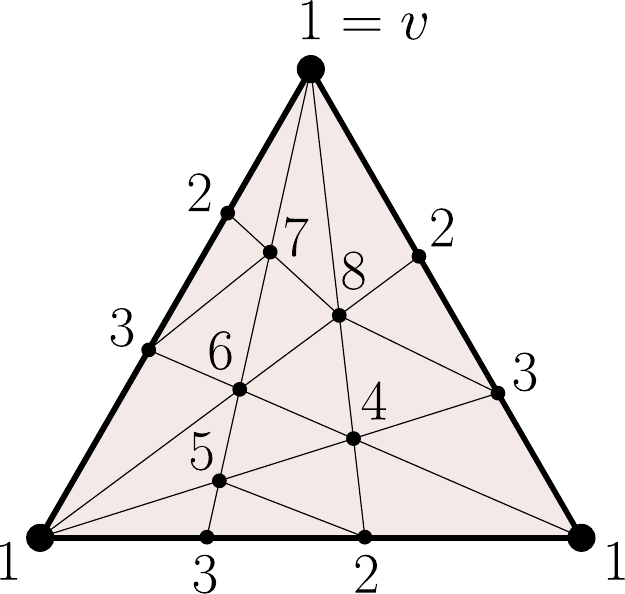}
    \end{center}
    (figure is from \cite{SANTAMARIAGALVIS2022103584}). The interior is glued to the edges such that the resulting map $S^1\to S^1$ is of multiplicity $1$, and so this  topological space is homotopic to a disk, and hence contractible. However, there is no free face of any \di, thus it is not collapsible.
\end{example}

\begin{lemma}\label{lemColapseComponents}
    A simplicial complex $C$ collapses onto \di \ $d$ if and only if all strong connectivity components of $C$ w.r.t \di\  $d+1$ collapse onto \di \ $d$.
\end{lemma}
\begin{proof}
    A particular elementary collapse $\sigma\subset \tau$ is possible if and only if $\sigma$ is free. This is determined only by faces containing $\sigma$, which are all above the \di \ of $\sigma$.
    If $\dim(\sigma)\ge d$, all faces containing it are in the same component by definition, and the possibility of this collapse is only affected by faces in the same strong connectivity component w.r.t \di \ $d+1$ as $\tau$.
    
    From here both directions are clear: any sequence of elementary collapses collapsing $C$ gives rise to a sequence collapsing any strong connectivity component (merely the subsequence of collapses happening in that component), and sequences of collapses for all components may be strung together to produce a collapse of $C$ onto \di \ $d$.
\end{proof}

\section{Lower Model}\label{chap:lower}
\begin{proposition}\label{propBudCostLower}
    The terms \textbf{budget} and \textbf{cost} are well defined for $\underline{X}(n;p_1,...)$ where $p_i=n^{-\alpha_i}$.
    More concretely, assume that a simplicial complex $A$ has $a_0$ vertices, $a_1$ edges and so on. Then $\lim_{n\to \infty}\log_n(E(A\subset \underline{X}))=a_0-\sum_{i=1} ^{\dim(A)}a_i\alpha_i$.
\end{proposition}
\begin{proof} 
    For a particular choice of an injective map $f$ from the vertices of $A$ to $a_0$ vertices of $\underline{X}$, the probability that this map extends to a map of complexes may be written as follows: 
    Define $B_k$ to be the event that $f$ sends all $k$-faces in $A$ to $k$-faces in $\underline{X}$. Then we are interested in $P\prs{\bigcap_{i=1} ^{\dim(A)}B_i}$.
    $$P\prs{\bigcap_{i=1} ^{\dim(A)}B_i}=P(B_1)P(B_2\mid B_1)P(B_3\mid B_2\cap B_1)...=\prod_{i=1} ^{\dim(A)} P\prs{B_i\middle\lvert\bigcap_{j=1} ^{i-1}B_j}$$
    by definition of conditional probability. But in the lower model $P(B_i\mid \bigcap_{j=1} ^{i-1}B_j)=p_i ^{a_i}=n^{-a_i\alpha_i}$. There are $\binom{n}{a_0}$ subsets of $[n]$ of size $a_0$, and for each such subset there are $a_0!$ possible maps $f$ (if $A$ has symmetries these maps might induce the same map on complexes, but this does not affect the asymptotics). 
    Therefore, 
    $$\binom{n}{a_0}n^{-\sum_{i=1} ^{\dim(A)}a_i\alpha_i}\le E(A\subset \underline{X})\le a_0!\binom{n}{a_0}n^{-\sum_{i=1} ^{\dim(A)}a_i\alpha_i}$$ 
    and so $\log_n(E(A\subset \underline{X}))\xrightarrow{n\to \infty} a_0-\sum_{i=1} ^{\dim(A)}a_i\alpha_i$.\par
    This generalizes the calculation from \cite{Fowler19} for the simplex, giving $E(\Delta^d\subset \underline{X})=\binom{n}{d+1}n^{-\sum_{i=1} ^{d}\binom{d+1}{i+1}\alpha_i}$ and a budget $\log_n(E(\Delta^d\subset \underline{X}))\xrightarrow{n\to \infty} d+1-\sum_{i=1} ^{d+  1}\binom{d+1}{i+1}\alpha_i$.\par
    As for expansion operations, since $\lim_{n\to \infty}\log_n(E(A\subset \underline{X}))$ depends only on the number of faces of each \di, an expansion operation $A\mapsto A\cup_Q \Delta^d$ increases the number of $i$-faces by $\binom{d+1}{i+1}-q_i$, where $q_i$ is the number of $i$-faces in $Q$. So the cost is well defined, and only depends on the type of expansion operation.
\end{proof}

\begin{remark}
    A corresponding result is also true for any $p_i$ asymptotically comparable to all $n^{-a}$. Lower order terms can be ignored, while if $p_i=o(n^{-a})$ for all $a$, we can say for the sake of \Cref{propBudCostLower} that $\alpha_i=\infty$, and similarly for $p_i=\omega(n^{-a})$, $\alpha_i=0$.
    
    Note that there are in general multiple choices of a growth process for a strongly connected complex. 
\end{remark}

\begin{remark}
    It is at this point that we may give an interpretation to the expressions $S^k _1(\{\alpha_i\}_i), S^k _2(\{\alpha_i\}_i)$ from \Cref{thmFowler}. $S^k _1(\{\alpha_i\}_i)
    \geq 1$ is the requirement that the vertex adding expansion operation of \di\ $k+1$ decreases the expected number of appearances of the subcomplex (as we prove in \Cref{opDecreasesE}).
    This means that $k+1$ \di al components are more common the fewer vertices they have.
    In terms of $k$-cohomology, this condition ensures that $k$-cycles are unlikely to be 'filled' by $(k+1)$-faces.
    
    Meanwhile, $S^k _2(\{\alpha_i\}_i)<k+2$ is the requirement that the complex $\partial(\Delta^{k+1})$ is likely to appear as a subcomplex.
    This is the smallest $k$-cycle, as well as a prerequisite for ($k+1$)-faces to appear at all.
\end{remark}

\begin{corollary}\label{opDecreasesE}
    Assume $S_1 ^k(\{\alpha_i\}_i)\geq 1$ and $m> k$. For a strongly connected $m$-\di al complex $D$, denote by $D'$ the result of some expansion operation performed on $D$, adding (among others) an $l$-\di al simplex. 
    If either $S_1 ^k(\{\alpha_i\}_i)> 1$ or the operation does not add vertices and $\alpha_l >0$, then $E(D')= o(E(D))$.
\end{corollary}
\begin{proof}
    By the proof of \Cref{propBudCostLower}, $E(D)$ is of the form $n^{f_0(D)-\sum_{i=1} ^\infty f_i(D) \alpha_i}$ (up to a constant factor), where $f_i(D)$ denotes the number of $i$-\di al faces of $D$.
    If the operation added a vertex, this costs at least $S^{m-1} _1(\{\alpha_i\}_i) -1$, and since $S^{m-1} _1(\{\alpha_i\}_i)\ge S^k _1(\{\alpha_i\}_i)>1$ the result follows. 
    Otherwise, we know that $f_0(D')= f_0(D)$, $f_i(D')\ge f_i(D)$ for all $i$ and that $f_l(D')>f_l(D)$, and so $n^{f_0(D')-\sum_{i=1} ^\infty f_i(D') \alpha_i}=o(n^{f_0(D)-\sum_{i=1} ^\infty f_i(D) \alpha_i})$ as $\alpha_l >0$.
\end{proof}
\begin{corollary}    \label{finIsoTypes}
        For $\{\alpha_i\}_i$ where $S^k _1(\{\alpha_i\}_i)>1$ and $m>k$, only a finite number of isomorphism types of strongly connected components $A$ have $E(A\subset\underline{X})\not\to 0$.
\end{corollary}
    \begin{proof}
     By \Cref{opDecreasesE}, all expansion operations decrease the asymptotics of $E(\bullet\subset\underline{X})$.
     In particular, if we find a number $v_0$ such that any $m$-\di al component $A$ with $|A|>v_0$ has $E(A\subset\underline{X})\to 0$, the claim will follow.

    Denote the budget by $b$, and the cost of a vertex adding operation by $c_m:=\sum_{i=1}^m\binom{m}{i}\alpha_i-1\ge S_1^k-1>0$. If $A$ has $a$ vertices, $E(A\subset\underline{X})\sim n^{b-(a-m-1)c_m}$. For $a$ such that $c_m>\frac{b}{a-m-1}$, the exponent will be negative, so the expectation will go to 0.

    Furthermore, the probability of any component to have more vertices than  $\frac{b}{c_m} + m + 1$ is bounded by the sum of probabilities of the components with exactly  $\ceil{\frac{b}{c_m} + m + 1}$ vertices (which is finite), and still tends to 0.
    \end{proof}
\begin{lemma}\label{lem:S2entS1}
    $S_2 ^d> d+2$ entails $S_1^{d-1}>1$.
\end{lemma}
\begin{proof}
    For all $1\le i\le d$, the coefficients of $\alpha_i$ in $S_2^d$ and $S_1^{d-1}$ are $\binom{d+2}{i+1},\binom{d}{i}$ respectively (the coefficients are $0$ outside of this range).
        $$\binom{d+2}{i+1}/\binom{d}{i}=\frac{(d+2)!i!(d-i)!}{d!(i+1)!(d+1-i)!}=\frac{(d+2)(d+1)}{(i+1)(d+1-i)}<d+2$$
        for all $1\le i\le d$, where the final inequality is true as $1\cdot(d+1)=(\frac{d+2}{2})^2-(\frac{d}{2})^2\le(\frac{d+2}{2})^2-(\frac{d}{2}-i)^2=(i+1)(d+1-i)$.
        All in all, coefficient-wise, $S_1^{d-1}\ge\frac{S_2^d}{d+2}$, so by assumption $S_1^{d-1}>1$. 
\end{proof}
We can now proceed to prove \Cref{thm:lowerCollapse}.

\begin{proof}[Proof of \Cref{thm:lowerCollapse}]
    By the proof of \Cref{propBudCostLower}, $S_2 ^d$ is the cost of $\partial(\Delta^{d+1})$, which has $d+2$ vertices. By assumption $d+2-S_2 ^d<0$, so by the first moment method a.a.s no simplexes of \di s $d+1$ and above appear. So it suffices to look at $d$-\di al strongly connected components.
    Since $S_2 ^d>d+2$, combining \Cref{lem:S2entS1} and \Cref{opDecreasesE} gives that all expansion operations of \di \ $d$ have a positive cost.
        
    For $d$-\di al strongly connected components to collapse to \di \ $d-1$, it is enough for a growth process to exist, such that every expansion operation has a free face, as in that case we may simply collapse the faces from last to first. 
    Since a.a.s no faces above \di\ $d$ are attached, the only expansions that have no free faces are those attached along $\partial(\Delta^{d})$.

    Assume $C$ is a strongly connected subcomplex in $\underline{X}$ with a growth process $\Delta^d=C_0\subset \dots \subset C_m=C$ such that $C=C_{m-1}\cup_{\partial(\Delta^{d})} \Delta^d$. This incurs no loss of generality, since any growth process featuring such an attachment at an intermediate stage can simply be truncated at that step.
    In $C_{m-1}$, the maximal faces of $\partial(\Delta^{d})$ come from distinct $d$-faces. Denote those faces by $f_0,...,f_d$, in order of addition in this growth process, and denote by $C_{t_i}$ the first stage containing $f_i$.

    By \Cref{propBudCostLower}, $\lim_n\log_n(E(C_{t_0}))\le d+1-\sum _{i=1}^d\binom{d+1}{i+1}\alpha_i$ (which is just the budget). 
    Further, 
    $$\lim_n\log_n(E(C_{t_1}))\le d+1-\sum _{i=1}^d\binom{d+1}{i+1}\alpha_i+1-\sum _{i=1}^d\binom{d}{i}\alpha_i$$
        as at a certain point a vertex adding operation must add the vertex of $\partial(\Delta^{d})$ not contained in $f_0$ (note- this claim is not contingent on this vertex being absent in $C_{t_0}$).
        Similarly, 
        $$\lim_n\log_n(E(C_{t_k}))\le d+2-\sum_{j=0}^{k-1}\sum _{i=1}^d\binom{d+1-j}{i+1-j}\alpha_i -\sum _{i=1}^d\binom{d+1-k}{i+1-k}\alpha_i,$$
        as we must add the $(j-1)$-face in $\partial(\Delta^{d})$ not contained in $\bigcup^{j-1}_{i=0}f_i$.

        All in all (under the convention that $\binom{b}{a}=0$ if $a<0$ or $a>b$): 
        $$\lim_n\log_n(E(C_{m-1}))\le d+2-\sum_{k=0}^{d}\sum _{i=1}^d\binom{d+1-k}{i+1-k}\alpha_i=d+2-\sum_{i=1}^{d}\sum _{k=0}^d\binom{d+1-k}{i+1-k}\alpha_i.$$
        For $i<d$, $\sum _{j=0}^d\binom{d+1-j}{i+1-j}=\binom{d+2}{i+1}$ by the hockey-stick identity. For $i=d$, $\sum _{j=0}^d\binom{d+1-j}{i+1-j}=d+1$. $C$ has one $d$-face more than $C_{m-1}$, and so 
        $$\lim_n\log_n(E(C))=-\alpha_d +\lim_n\log_n(E(C_{m-1}))\le d+2-S_2^d<0.$$
        We conclude $E(C)\to 0$, which finishes the proof by Markov's inequality.
\end{proof}

\subsection{Special Cases}
\begin{example}\label{exm:lower}
    The lower multiparametric model $\underline{X}(n;p_1,...)$ specializes to the Linial–Meshulam model $Y_d(n;p)$ by substituting $p_1,...,p_{d-1}=1, p_d=p, p_{d+1},...=0$.
    If $p=n^{-a}$ (and $1=n^0$), this means that $S_2^d=0+0+...+(d+2)a$, and by \Cref{thm:lowerCollapse} we get that $Y_d(n;p)$ collapses to \di\ $d-1$ if $a>1$.

    This provides an alternative proof of the same results in \cite{CCFK12} (for the 2-\di al case) and in \cite{AronshtamLinialLM13} (for the general case), although not addressing the $p\sim 1/n$ case.
    Further, it reinforces the claim in \cite{AronshtamLinialLM13} that simplex boundaries are the only obstruction for collapsibility (as this exactly corresponds to $S_2^d$).
\end{example}
\begin{example}\label{ex:clique}
    The clique/flag complex $X(G(n;p))$ coincides with $\underline{X}(n;p,1,1,...)$. Therefore, denoting $p=n^{-b}$, \Cref{thm:lowerCollapse} applies to any \di\ $d$ such that $S_2^d=\binom{d+2}{2}b+0>d+2$, or in other words for $b>\frac{2}{d+1}$.
    
    This is worse than the bound $b>\frac{1}{d+1}$ obtained in \cite{MALEN23}, but this is to be expected- the proof of that bound relies on the following result:
    \begin{theorem}[Theorem 3.1 in \cite{MALEN23}]
Fix $k \ge 0$. Let $X$ be a finite clique complex such that every strongly connected, pure $(k+1)$-dimensional subcomplex $S \subseteq X$ contains at least one vertex $v$ with $\deg_S(v) \le 2k + 1$.
Then $X$ collapses to \di\ $k$.
\end{theorem}
This theorem is reliant on $X$ being a clique complex- the complexes $\partial(\Delta^{k+2})$ satisfy the conditions required for $S$, but not the conclusion.  
Indeed, these simplex boundaries may arise as induced subcomplexes of $\underline{X}(n;n^{-b},p_2,p_3,...)$ if $p_2,p_3...=o(1)$, explaining the necessity of the condition that the $\alpha$'s are positive in \Cref{thmFowler}.
\Cref{thm:lowerCollapse} would still apply if $p_2,p_3...$ tend to zero slower than $n^{\beta}$, for $\beta$ sufficiently small (as $1/\log n$, for instance).
\end{example}

\section{Upper model}\label{chap:Upper}
In this section we transition to thinking about the upper multiparametric model $\overline{X}(n;p_1,...)$ where $p_i=n^{-\alpha_i}$ and $p_j=0$ for $j>r$, some maximal \di.
Unlike in the lower model, the components we wish to collapse may be of a significantly greater \di\ than what they collapse to, so we need a slight change in definitions:
\begin{definition}
    Let $C\subset C'$ be simplicial complexes. $C$ is called an \textbf{honest subcomplex} if no two maximal faces of $C$ are contained in the same maximal face of $C'$.
    If this is the case, we denote it by $C\subset_h C'$. Clearly $E(C\subset_h \overline{X})\le E(C\subset \overline{X})$. Replacing $\subset$ by $\subset_h$ in the definition of cost in \Cref{defBudgetCost}, we call this quantity the \textbf{honest cost} (The "honest budget" is identical to the budget as the honesty condition is trivial in the case of a single simplex).
\end{definition}
\begin{remark}
    Strong connectivity components w.r.t any \di \ are honest subcomplexes. The reason for using this definition rather than the components themselves is for ease of calculation.

    In the lower model, the non-honest cost decreases predictably with every expansion operation. This is no longer true in the upper model (as a complex with many maximal faces may be contained in a single simplex of a greater \di).
    In contrast, the honest cost does decrease predictably under honest containment, as we will see in \Cref{lem:upperBudgetCost}.
\end{remark}

Recall from \Cref{FNNotation} that $\beta_i:=i+1-\alpha_i$, $\beta:=\max\{\beta_i\}$, $l=\lfloor\beta \rfloor$. Recall as well that $\gamma_k:=\max_{i\ge k}\{\beta_i\}$ and $e_k:=\gamma_k-k$.

\begin{lemma}\label{lem:upperBudgetCost}
    Let $Z$ be a strongly connected complex w.r.t \di \ $l+1$. In $\overline{X}(n;p_1,...,p_r)$, $p_i=n^{
    -\alpha_i}$ we have the following:
    \begin{enumerate}
        \item $E(Z\subset \overline{X})=O(n^\beta)$. In particular, the budget is less than $l+1$.
        \item Let $Z'=Z\cup \tau, \tau\simeq\Delta^m, m>l$ be the result of an operation expanding $Z$ (w.r.t \di\ $l+1$) which increases the number of maximal faces. Then $E(Z'\subset_h \overline{X})=O(E(Z\subset_h \overline{X})n^{e_{l+1}+v+l-m})$, where $v$ is the number of new vertices added by the operation.
        In other words, the honest cost of such an operation is at least $m-e_{l+1}-v-l$ (note that $e_{l+1}<0$ and that $m\ge v+l$).
    \end{enumerate}
\end{lemma}
\begin{proof} \begin{enumerate}
    \item Consider a particular injective mapping $g$ of the vertices of $\sigma=\Delta^d, d\ge l+1$ to $[n]$. $g(\sigma)$ is a face in $\overline{X}$ if and only if $g(\sigma)$ is contained in a face in $X$ (the hypergraph).
Denote by $J_i$ the event that $g(\sigma)$ is contained in an $i$-\di al face in $X$. Then 
$$P(J_i)=\binom{n-d-1}{i-d}n^{-\alpha_i}-\sum_{j=0}^{i-d-1}\binom{n-d-1}{2(i-d)-j}n^{-2\alpha_i}+\dots=\theta(n^{i-\alpha_i-d})$$
where we use inclusion-exclusion, and the fact that the appearance of different simplexes in $X$ is independent. Since $\max(i+1-\alpha_i)=\beta<l+1$, we conclude, assuming $d\ge l+1$, that $i-\alpha_i-d<-1$, and so the first term dominates asymptotically.
Further, $c\sum_{j=0} ^\infty\frac{1}{n^j}=\theta(c)$, so the sum is of the claimed asymptotic class.
Using this we get 
$$P\prs{\bigcup_{i=d} ^r J_i}=\theta\prs{\sum_{i=d} ^r n^{i-\alpha_i-d}-\sum_{i\ge j\ge d} ^r n^{i+j-\alpha_i-\alpha_j-2d}+ \sum_{i\ge j\ge k\ge d}\dots}$$
where we use the same arguments again, and conclude 
$$P\prs{\bigcup_{i=d} ^r J_i}=\theta\prs{\sum_{i=d} ^r n^{i-\alpha_i-d}}.$$
By the same argument as in \Cref{propBudCostLower}, there are $\theta(n^{d+1})$ possible $g$'s, and so:
\begin{equation}\label{bugdetUpper}
    E(\Delta^{d}\subset \overline{X})=\theta(n^{d+1}(n^{-\alpha_d}+n^{1-\alpha_{d+1}}+...+n^{r-d-\alpha_r}))=\theta(n^{\max_{i\ge d}\{ i+1-\alpha_i\}})=\theta(n^{\gamma_d}).
\end{equation}

 But $\gamma_d\le \beta < l+1$, giving the budget.
 \item Denote $E(Z\subset_h\overline{X})=\zeta$, and suppose $Z'$ has $v$ vertices more than $Z$. Let $g$ be an injective map from the vertices of $Z'$ to $[n]$. $g(Z')$ is an honest subcomplex of $\overline{X}$ if and only if all the following are satisfied:
 \begin{enumerate}[label=\alph*)]
     \item $g(Z)$ is an honest subcomplex.
     \item $g(\tau)$ is a face in $\overline{X}$.
     \item No face of $\overline{X}$ contains both $g(\tau)$ and a maximal face of $g(Z)$.
 \end{enumerate}
 Let $\crb{Z_i}$ be the maximal faces of $Z$. Then by the same arguments as above
 $$P\prs{g(Z')\subset_h\overline{X}\mid g(Z)\subset_h\overline{X}}=$$$$\sum_{i=m} ^r \prs{\binom{n-m-1}{i-m}-\sum_{Z_j}\binom{n-|Z_j\cup \tau|}{i-|Z_j\cup \tau|+1}}n^{-\alpha_i}+o(n^{i-m-\alpha_i})=\theta(n^{\gamma_m-m-1}).$$
The difference of binomial coefficients is to exclude maximal faces containing both $g(\tau)$ and one of the $Z_j$ (containing more than one is excluded by the condition that $g(Z)$ is an honest subcomplex), and the asymptotic is achieved by the observation that $\sum_{Z_j}\binom{n-|Z_j\cup \tau|}{i-|Z_j\cup \tau|+1}=o\prs{\binom{n-m-1}{i-m}}$ as well.
$$P(g(Z')\subset_h\overline{X})=P(g(Z')\subset_h\overline{X}\mid g(Z)\subset_h\overline{X})\cdot P(g(Z)\subset_h\overline{X})=\theta\prs{P(g(Z)\subset_h\overline{X})n^{\gamma_m-m-1}}.$$
$E(Z\subset_h\overline{X})=\sum_{g|_Z}P(g(Z)\subset_h\overline{X})$ and similarly $E(Z'\subset_h\overline{X})=\sum_{g}P(g(Z')\subset_h\overline{X})$.
Therefore 
$$E(Z'\subset_h\overline{X})=\sum_{g}P(g(Z')\subset_h\overline{X})=\theta\prs{n^{\gamma_m-m-1}\sum_{g}P(g(Z)\subset_h\overline{X})}=$$
$$\theta\prs{n^{\gamma_m-m-1+v}\sum_{g|_Z}P(g(Z)\subset_h\overline{X})}=\theta\prs{n^{\gamma_m-m-1+v}E(Z\subset_h\overline{X})}.$$
Recall from \Cref{FNNotation} that $e_m=\gamma_m-m$, and that $e_m\le l+1-m+e_{l+1}$, and we get the desired result.
    \end{enumerate}
\end{proof}
Part 2 of the lemma should be understood in the following way- the new simplex must intersect with the old complex at least at $l+1$ vertices in order to maintain strong connectivity.
If all other vertices are new- the operation costs $e_{l+1}$, and we pay $1$ for any additional vertex of intersection. We formalize the concept in the next subsection.

\subsection{Redundancy}
\begin{definition}\label{def:redun}
    Let $Z_0\subset Z_1\subset...\subset Z$ be a growth process of a simplicial complex, strongly connected w.r.t \di \ $k$. The expansion operation $Z_i\cup _{C}\Delta^d= Z_{i+1}$ is said to have \textbf{redundancy} $t_i$ if $|Z_{i+1}|-|Z_i|=d+1-k-t_i$, i.e $t_i=d+1-(|Z_{i+1}|-|Z_i|)-k$, meaning $\Delta^d$ is attached along $t_i+k$ "old vertices"- vertices in $Z_i$. 
    The redundancy of $Z$ is said to be $R(Z):=\sum_i t_i$.
\end{definition}
\begin{remark}
    The redundancy of a complex w.r.t \di \ $k$ does not depend on the choice of the growth process. Indeed, denoting $ Z_{i+1}=Z_i\cup _{C}\sigma_{i+1}$, 
    $$|Z|=|Z_0|+\sum_{i=1}\prs{\dim(\sigma_{i})+1-k-t_i}=\sum_{\sigma\in M(Z)}(\dim(\sigma)+1)-(|M(Z)|-1)k-R(Z).$$
    Therefore
    \begin{equation}\label{eqDefRedun}
        R(Z)=\sum_{\sigma\in M(Z)}(\dim(\sigma)+1)-(|M(Z)|-1)k-|Z|
    \end{equation}
    which is independent of the process. Note- the above formula also does not require $Z$ to be strongly connected, and therefore can be used as the definition in this case. \textbf{Note: only faces of \di \ greater or equal to $k-1$ count as maximal for this purpose}. \par
    We may now reinterpret \Cref{lem:upperBudgetCost} as saying that an expansion operation $op$ has a cost of $R(op)-e_{l+1}$.
    The redundancy may be thought of as a measure of the complexity of a complex.
    Redundancy $0$ complexes are akin to "trees", as any new face adds the maximal number of vertices keeping it strongly connected w.r.t \di \ $k$.
\end{remark}
\begin{lemma}\label{lem:redunMonotonic}
    Let $C\subset Z$ be an induced subcomplex. Then (w.r.t $k$) $R(C)\le R(Z)$.
\end{lemma}

\begin{proof}
The general case will follow from the case where $C$ and $Z$ differ by a single vertex:
Suppose $C$ is the induced subcomplex on $Z\setminus\crb{w}$, for some vertex $w$.  

We observe how the expressions in \Cref{eqDefRedun} differ between $C$ and $Z$: Obviously $|C|=|Z|-1$.
If $|M(C)|=|M(Z)|$, $\sum(\dim(\sigma)+1)$ decreases by at least $\deg_M(w)$- the number of maximal faces containing $w$. This is a non-empty set, and so in this case $R(C)\le R(Z)$.

If $|M(C)|<|M(Z)|$ (it cannot increase as different maximal faces in $C$ come from different maximal faces in $Z$), suppose $\crb{\sigma_j\setminus w}_{j\in J}$ are no longer maximal.
Then $k|M(C)|=k\prs{|M(Z)|-|J|}$, but 
$$\sum_{\sigma\in M(C)}(\dim(\sigma)+1)\le \sum_{\sigma\in M(Z)}(\dim(\sigma)+1)-\sum_{j\in J} (\dim(\sigma_j)+1)$$
(where the inequality is due to the possible reductions in \di s).
$\dim(\sigma_j)+1\ge k$ for all $j$, and comparing with \Cref{eqDefRedun} grants the result in this case as well.
\end{proof}

\begin{lemma}\label{lem:gluingcollapse}
    Let $C'=C\cup_G \sigma$, where $\sigma$ is a simplex. If $G$ has at most $c$ maximal faces
    of \di\ greater or equal to $c-1$, $C'$ collapses onto $C$ and perhaps some faces of \di\ below $c$.
\end{lemma}
\begin{proof} For two faces $f_1,f_2$, denote by $f_1*f_2$ their \textbf{join}, i.e the simplex composed of all of their vertices.

By \Cref{lemColapseComponents}, only faces of \di\ greater or equal to $c-1$ affect collapse to this \di, and so we may simplify the assumption to $G$ having at most $c$ maximal faces.\par
    We proceed by induction on the number of maximal faces in $G$: For a single maximal face the claim is clear.
    Assume we proved the claim for any $G$ with less than $k\le c$ maximal faces.
    WLOG we assume all vertices in $\sigma$ are in $G$ (any new vertex is free, and so may be gotten rid of by collapses). We proceed by induction on the number of vertices $g$ in $G$. For $g< k+1$  the claim is trivial. We divide into cases:
    \begin{enumerate}
    \item If there exists a vertex $v$ common to all maximal faces of $G$, the claim follows as we may simply perform collapses towards $v$ (i.e erasing all pairs of the form $(\tau,\tau\cup \crb{v})$, where $v\not\in\tau$ and $\tau\not\subset C$).
        \item If each vertex has exactly 1 maximal face of $G$ not containing it, then either there exists a pair $(f,v)$ ($f$- maximal, $v$-vertex) where $f\cup v \ne G$, or all maximal faces of $G$ are facets of $\sigma$.
    \begin{enumerate}
        \item In the first case, $G\cup (f*v)$ has a centre $v$ (as in the previous case), and so $C'$ collapses to $C\cup (f*v)$. $(f*v)$ has fewer vertices than $G$, and $(f*v)\cap G$ has no more maximal faces (as it is an induced subcomplex of $G$). The claim follows by induction.
    \item  The second case cannot happen, as the only set of facets without a centre is the complete one, which would have more than $k$ elements. \end{enumerate}

    \item Suppose there exists a vertex $w$ not contained in at least 2 maximal faces of $G$-$f_1,f_2$. $G\cup (f_1 *f_2)$ has fewer maximal faces than $G$, so $C'$ collapses by induction to $C\cup (f_1 *f_2)$ and remnants of \di\ less or equal to $k-2$.
    $G\cap (f_1 *f_2)$ has less vertices than $G$ (as $w$ is not there), and the claim follows again by induction. \qedhere
    \end{enumerate}
\end{proof}

\begin{lemma}\label{lem:redunColapse}
    Let $Z$ be a component strongly connected w.r.t \di \ $k$ in a complex $Z'$. If $R(Z)\le k-1$, then $Z$ collapses onto \di \ $k-1$ (i.e disappears as a component).
\end{lemma}
\begin{proof}
By \Cref{lemColapseComponents} we may proceed to look only at $Z$.
For $Z$ with $R(Z)\le k-1$, denote each expansion operation in a growth process of $Z$ as $ Z_{i+1}=Z_i\cup _{C_i}\sigma_{i}, i\ge 0$.

It suffices to prove that every such $C_i$ has at most $k$ maximal faces. Indeed, in such a case \Cref{lem:gluingcollapse} will imply that $Z_{i+1}$ collapses onto $Z_i\cup T$, where $T$ is of \di \ at most $k-1$.
By \Cref{lemColapseComponents}, $T$ does not affect the collapsibility of $Z_i$, and so the argument will work for all $i$, proving $Z$ collapses to \di \ $k-1$.
 
 For the sake of contradiction, assume that there exists $I$ such that $C_I$ has $k+1$ or more maximal faces of \di\ greater or equal $k-1$. Each maximal face of $C_I$ is contained in at least one maximal face of $Z_I$. Denote the first such face for each maximal face of $C_I$ by $\sigma_{i_j}, j\ge 1$, where  $i_1 < i_2 < \dots$ .

 Recall that $R(Z_{i+1})=R(Z_i)+|C_i|-k$, and $R(Z_0)=0$. Consider the expansion operations $Z_{i_j}\mapsto Z_{i_j+1}, j\ge 2$. Denote $\delta_j=|C_I\cap Z_{i_j+1}|-|C_I\cap Z_{i_j}|$.
 Then $|C_I|-k\ge \sum_{j\ge 2}\delta_j$ (as $|C_I\cap \sigma_{i_1}|\ge k$ by assumption), and therefore $R(Z_{I+1})\ge R(Z_I)+\sum_j\delta_j$.

Let $\epsilon_j = R(Z_{i_j+1}) - R(Z_{i_j})$ denote the redundancy added at step $i_j$. If the expansion $Z_{i_J} \mapsto Z_{i_J+1}$ does not add any new vertices to $C_I$ (meaning $\delta_J = 0$), then the intersection $\sigma_{i_J} \cap C_I \cap \prs{\bigcup^{j<J} \sigma_{i_j}}$ must consist of at least two maximal faces (otherwise the intersection is an existing face of $ C_I\cap \prs{\bigcup^{j<J} \sigma_{i_j}}$), therefore $R(Z_{i_J+1})-R(Z_{i_J})=\epsilon_J\ge 1$.

 All in all, $R(Z)\ge R(Z_{I+1})\ge \sum_{j=2}(\delta_j+\epsilon_j)\ge k$, contradicting our assumption.
 Thus $C_i$ has at most $k$ maximal faces, as desired.
 \end{proof}

\begin{remark}
    The bound in \Cref{lem:redunColapse} is tight. This is easily seen by plugging $\partial(\Delta^{k+1})$ into \Cref{eqDefRedun}.
\end{remark}
\subsection{Proof of \Cref{thm:uppercollapse}}
We now have set the table to prove \Cref{thm:uppercollapse}.
\begin{proof}
    By \Cref{lemColapseComponents}, collapsibility to a certain \di \ is determined by strong connectivity components w.r.t the \di\ above it. Since strong connectivity components are always honest subcomplexes, it will be enough to prove that all honest subcomplexes of $\overline{X}$ strongly connected w.r.t \di \ $l+1$ collapse to \di \ $l$.

    In \Cref{lem:upperBudgetCost} we saw that an expansion operation w.r.t \di \ $l+1$ has an honest cost of at least $-e_{l+1}$, and we start with a budget of less than $l+1$. 
    From this we can deduce by arguments identical to those of \Cref{finIsoTypes} that only a finite number of isomorphism types of complexes strongly connected w.r.t \di \ $l+1$ may a.a.s appear as honest subcomplexes of $\overline{X}$, and those have less than $l+2+\frac{l+1}{|e_{l+1}|}$ vertices.\par
     By \Cref{lem:upperBudgetCost} we may also exclude all strong connectivity components with redundancy greater or equal to $l+1$, as those would cost more than the budget.
    But by \Cref{lem:redunColapse}, a component strongly connected w.r.t \di \ $l+1$ with redundancy of at most $l$ collapses onto \di \ $l$, finishing the proof.     
\end{proof}

\begin{corollary}\label{cor:wedgeSpheres}
    Under the conditions of \Cref{thm:uppercollapse} and assuming $\beta\not \in \N$, the integral homology of $\overline{X}$ is concentrated in \di\ $l$ (by \Cref{FNResults}) and is free.
    If $l>2$, $\overline{X}$ is also simply connected, and consequently
    is homotopic to a wedge of spheres of \di\ $l$ (this is a folklore result in algebraic topology, proven by combining the Hurewicz and Whitehead theorems). 
\end{corollary}

\section{Open Problems}\label{discus}

\begin{itemize}
    \item As we mentioned after the statement of \Cref{thm:uppercollapse}, the upper model, after the collapse, appears to be similar to the Linial-Meshulam model with intersecting faces being positively correlated.
    It might be interesting to investigate in $\overline{X}$ various phenomena prominent in the study of $Y_d(n,p)$, such as the shadow or expansion.
    \item The point above also suggests an approach to an underexplored area in random complexes- complexes that are symmetric (such as all models mentioned in this paper), but are not built from independent parts. 
    This calls to mind things like the Ising model or its generalizations- the Random Cluster Model of random graphs (see \cite{FortuinKasteleyn72}), or the Plaquette Random-Cluster Model in cubical complexes (see \cite{DuncanSchweinhart25}). In these models, however, the probabilities are assigned based on topological properties of the entire complex, rather than individual correlations.
    \item  \Cref{cor:wedgeSpheres} does not apply when $\beta\in \N$. This is where the complex "transitions" between the regime where it has complete ($l-2$)-skeleton to one where it has complete $(l-1)$-skeleton, and so might have non-trivial homologies in two \di s simultaneously. 
    This, as always, begs the question of whether it is a wedge of spheres in that critical case as well.
    \item In the lower model we did not address the case where $S_2^d=d+2$ for some $d$, which is also unaddressed in Fowler's \Cref{thmFowler}. 
    It is clear that the answer in this case will depend on lower order terms, since if $E(\partial(\Delta^{d+1}))\to \infty$ and $p_{d+1}$ is sufficiently small, there will be simplex boundaries generating homology and preventing the collapse, while $E(\partial(\Delta^{d+1}))\to 0$ enables us to use the present proof as is.
    The case where $E(\partial(\Delta^{d+1}))$ tends to a positive constant would also depend on whether $p_{d+1}=1$ or not, giving outcomes with probabilities not necessarily tending to $0$ or $1$ in the latter case.
    \item \Cref{lem:redunColapse} should be useful in other questions concerning collapsibility. The author also believes that these are not the only possible uses of redundancy. Certainly, a complex with almost maximal redundancy is close to the complete skeleton of a simplex, but what is the minimal redundancy for which a complex cannot be collapsible, or have non-trivial homology? If these are different, can we characterize the complexes "in between", or construct a model generating them?
\end{itemize}

\printbibliography
\end{document}